\documentclass[a4paper,12pt]{article}
\usepackage{amsmath,amsthm,amsfonts,amssymb,graphicx}
\newtheorem{theorem}{Theorem}
\newtheorem{lemma}[theorem]{Lemma}
\newtheorem{corollary}[theorem]{Corollary}

\newtheorem{conjecture}[theorem]{Conjecture}
\newtheorem*{conjecture*}{Conjecture}

\newcommand{\R}{{\mathbb R}}
\newcommand{\rodl}{R\"odl}

\newcommand{\Q}{{\mathbb Q}}
\newcommand{\A}{{\mathbb A}}

\begin{document}
\title{Transitive Sets and Cyclic Quadrilaterals}
\author{Imre Leader\footnote{Department
of Pure Mathematics and Mathematical Statistics,
Centre for Mathematical Sciences,
Wilberforce Road,
Cambridge CB3 0WB,
England.} \footnote{{\tt I.Leader@dpmms.cam.ac.uk}}\and 
Paul A.~Russell\footnotemark[1] \footnote{{\tt P.A.Russell@dpmms.cam.ac.uk}}
\and Mark Walters \footnote{School of Mathematical Sciences, Queen Mary,
University of London, London E1 4NS, England} 
\footnote{\tt m.walters@qmul.ac.uk}}
\maketitle
\begin{abstract}
Motivated by some questions in Euclidean Ramsey theory, our aim in
this note is to show that there exists a cyclic quadrilateral that
does not embed into any transitive set (in any dimension). We show
that in fact this holds for almost all cyclic quadrilaterals, and we
also give explicit examples of such cyclic quadrilaterals. These are
the first explicit examples of spherical sets that do not embed into
transitive sets.
\end{abstract}

\subsubsection{Introduction}
A finite set $X$ in some Euclidean space $\R^n$ is called {\it Ramsey}
if for every positive integer $k$ there exists $d$ such that whenever
$\R^d$ is $k$-coloured it contains a monochromatic set congruent to
$X$. A famous question of Erd\H{o}s, Graham, Montgomery, Rothschild,
Spencer and Straus~\cite{egmrss} (made into a conjecture by Graham
in~\cite{conj}) asks if the Ramsey sets are precisely the spherical
sets (where a set is {\it spherical} if it lies on the surface of a
sphere).

In~\cite{us}, a `rival' conjecture is made: that a set is Ramsey if
and only if it is a subset of a finite transitive set (in some
dimension) -- here we say that a set is {\it transitive} if its
isometry group acts transitively on it (or equivalently, loosely
speaking, if all points of the set look the same). It is not obvious
that these conjectures are different: certainly every finite
transitive set $X$ is spherical (for example, because all of its
points must lie on the surface of the minimal sphere containing $X$),
but does every finite spherical set embed into a transitive set (in
some dimension)? This is answered in the negative in~\cite{us}, where
it is shown that, for $k \geq 16$, almost every cyclic $k$-gon
does not embed into a (finite) transitive set.

However, in \cite{us} we conjectured that it was unnecessary to
go as far as 16-gons to see this phenomenon: that in fact there exists
a cyclic quadrilateral that does not embed into a transitive
set. (This would be the smallest possible case, as all triangles do
embed into transitive sets -- see e.g.~\cite{triangles} or~\cite{us}).

Our aim in this note is to prove this. In fact, we show that almost
every cyclic quadrilateral does not embed into a transitive set. Our
proof allows us to give an explicit example of such a cyclic
quadrilateral. This is the first explicit example of a spherical set
that does not embed into a transitive set.

Curiously, it seems that the right approach is to focus on the linear
properties of the quadrilateral, rather than its metric
properties. Our main result is as follows.

\begin{theorem}\label{t:main}
Let $x$, $y$, $z$ and $w$ be four distinct points lying on a circle
such that
\[
w = z + \alpha (x-z)  + \beta (y-z),
\]
where $\alpha\not=1$ and $\beta$ is transcendental over
$\Q(\alpha)$. Then 
$xyzw$ does not embed into a transitive set.
\end{theorem}
\noindent%
We remark that the condition $\alpha \ne1$ is necessary.  Indeed, for
any $\beta$ there is a trapezium with parameters $1$ and $\beta$ (in
other words, with this value of $\beta$ and with $\alpha=1)$ which
embeds in a transitive set with symmetry group $D_8$. We note also
that the `cyclic' condition is trivially redundant as all
quadrilaterals that embed in a transitive set are cyclic.

Since `many' pairs $(\alpha,\beta)$ can occur as parameters of cyclic
quadrilaterals -- for example there exists such a cyclic quadrilateral
for every $\alpha$ and $\beta$ sufficiently close to $1$ -- it is
routine to verify from this that almost every cyclic quadrilateral
does not embed into a transitive set.  It is also possible to give an
explicit example -- even one with some symmetry, such as a kite.
\begin{corollary}\label{c:main}
The cyclic quadrilateral with vertices
\[
(-1,0),(1,0),(a,\sqrt{1-a^2}),(a,-\sqrt{1-a^2}),
\] 
where $a$ is transcendental, does not embed into any transitive
set.\qed
\end{corollary}
\noindent%
Corollary~\ref{c:main} follows from Theorem~\ref{t:main} upon taking $z=(-1,0)$,
$y=(1,0)$, $x=(a,\sqrt{1-a^2})$ and $w=(a,-\sqrt{1-a^2})$.

As explained above, this gives us an explicit spherical set that we
conjecture is not Ramsey.
\begin{conjecture}
Let $-1<\alpha<1$ be transcendental. Then the cyclic quadrilateral
with vertices
  \[
  (-1,0),(1,0),(a,\sqrt{1-a^2}),(a,-\sqrt{1-a^2})
  \] 
is not Ramsey.
\end{conjecture}

\subsubsection{Proof of Theorem~\ref{t:main}}
It will be convenient to use the term \emph{quadrilateral} to denote
any set of four coplanar points, whether or not they are distinct. We say
that it is \emph{trivial} if all four points are coincident. Note
that a non-trivial quadrilateral may still have some points
coincident.

Suppose that $xyzw$ is any quadrilateral. It may be the
case (for example whenever $x,y,z$ are not collinear) that there exist
(not necessarily unique) $\alpha,\beta$ such that
$w=z+\alpha(x-z)+\beta(y-z)$.  In this case we say that the
quadrilateral has \emph{parameters} $\alpha,\beta$. 

One reason why this parameterisation in terms of linear rather than
metric properties is useful is as follows. Suppose we have a
non-trivial quadrilateral in a vector space $V \oplus W$. Then the
projections onto $V$ and $W$ have the same parameters.  Moreover at
least one of these projections is non-trivial.  (Note that both
projections may have coincident points even if the original
quadrilateral does not, for example if the original quadrilateral is
a rectangle.) This will allow us to focus on irreducible representations.

The following is our key result.

\begin{lemma}\label{l:transitive}
  Let $G$ be a finite group generated by elements $g$, $h$ and $k$,
  and let $A$, $B$ and $C$ be the maps corresponding to $g$, $h$ and
  $k$ respectively in a non-trivial irreducible real orthogonal
  representation of $G$.  Then the collection of pairs
  $(\alpha,\beta)$ that can occur in quadrilaterals of the form
  $A(w),B(w),C(w),w$ for some $w\ne 0$ is exactly the zero set of a polynomial
  $P(\alpha,\beta)$ with algebraic coefficients. Moreover, for any
  fixed $\alpha\not=0,1$ the polynomial $P$ viewed as a polynomial in
  $\beta$ is not identically zero.
\end{lemma}
\begin{proof}
Let $d$ be the dimension of the representation. Suppose that that
there exists non-zero $w\in\R^d$ with
\begin{equation}\label{e1}
w=C(w)+\alpha(A(w)-C(w))+\beta(B(w)-C(w)).
\end{equation}
Then
\[
\left(\alpha A+\beta B+(1-\alpha-\beta)C-I\right)(w)=0
\]
and so, in particular, the linear map
\[
L(\alpha,\beta)=\alpha A+\beta B+(1-\alpha-\beta)C-I
\]
is singular. Conversely, if $L(\alpha,\beta)$ is singular then there
exists non-zero $w$ with $L(\alpha,\beta)(w)=0$; that is, with $w$
satisfying (\ref{e1}).

Obviously $P(\alpha,\beta)=\det(L(\alpha,\beta))$ is a polynomial in
$\alpha$ and $\beta$.  It is well known that any representation of $G$
over $\R$ is equivalent to a representation with all matrix entries
algebraic. Since the polynomial $\det(L(\alpha,\beta))$ is the same
for equivalent representations, this implies that $P(\alpha,\beta)$
has algebraic coefficients.

Suppose $\alpha\not=0,1$. Since $P$ is a polynomial, to complete the
proof we just need to give one value of $\beta$ for which $P$ is
non-zero or, equivalently, for which there exists non-zero $w$ such
that $L(\alpha,\beta)(w)=0$. This is the same as saying that there is
no non-zero $w$ for which the quadrilateral $w$, $A(w)$, $B(w)$,
$C(w)$ has parameters $\alpha$ and $\beta$.

We note first that a quadrilateral $w,A(w),B(w),C(w)$ cannot be trivial; that is
\begin{equation}\label{e2}
w=A(w)=B(w)=C(w)
\end{equation}
cannot occur for any (non-zero) $w$. Indeed, this would imply that the
representation is either reducible or one-dimensional. The former
cannot occur by hypothesis; the latter cannot occur since then
$A=B=C=I$, which is ruled out as the representation is
non-trivial.

Rather surprisingly, we will be able to choose $\beta$ in such a way
that the only possible (non-trivial) quadrilaterals with parameters
$\alpha$ and $\beta$ that occur in this way are not even convex (so
manifestly are not cyclic and so do not embed in a transitive set).
We split into the following three cases: $\alpha<0$, $0<\alpha<1$,
and $\alpha>1$.  

First, if $0<\alpha<1$ then let $\beta=(1-\alpha)/2$. Then (\ref{e1})
becomes
\[
w=\alpha A(w)+\frac{1-\alpha}{2}B(w)+\frac{1-\alpha}{2}C(w).
\]
Since $\alpha>0$ and $(1-\alpha)/2>0$, this
implies~(\ref{e2}), a contradiction.

Secondly, if $\alpha<0$ then let $\beta=\alpha$. Then (\ref{e1}) becomes
\[
(-\alpha) A(w)+(-\alpha)B(w)+ w= (1-2\alpha)C(w).
\]
Since $-\alpha$ and $1-2\alpha$ are both positive, once again this
implies~(\ref{e2}).

Finally, if $\alpha>1$ then let $\beta=(1-\alpha)/2$.  By a similar
argument to the previous case, this again implies~(\ref{e2}). This
completes the proof.
\end{proof}

\noindent%
{\it Proof of Theorem 1.}  
Suppose that we have a transitive set $T$
with symmetry group $G$ and a non-trivial quadrilateral $xyzw$ in $T$
with
\begin{equation}\label{e3}
 w=z+\alpha(x-z)+\beta(y-z).
\end{equation}
We aim to show that $\beta$ is algebraic over $\Q(\alpha)$.  Let
$A,B,C$ be elements of $G$ that map $w$ to $x,y,z$ respectively, and
let $G'$ be the group generated by $A,B$ and $C$.

We shall use Lemma~\ref{l:transitive}, but first observe that if
$\alpha=0$ then $w$, $z$ and $y$ are collinear and so at least two must
coincide, contradicting the hypothesis of the theorem.

Suppose that the representation of $G'$ is reducible as $V\oplus W$.
Then (\ref{e3}) holds for the projections onto $V$ and $W$, and in at
least one of these cases the projected quadrilateral is non-trivial.

It follows that the parameters $(\alpha,\beta)$ occur for some some
(non-trivial) irreducible representation of $G'$.  This implies that
$P(\alpha,\beta)=0$ for some polynomial $P$ as in the conclusion of
Lemma~\ref{l:transitive}. Writing $R(Y)$ for $P(\alpha,Y)$ and $\A$
for the field of algebraic numbers, we have that the polynomial $R$
has coefficients in $\A(\alpha)$ and is not identically zero.
Moreover, $R(\beta)=0$. Hence $\beta$ is algebraic over $\A(\alpha)$
and so is algebraic over $\Q(\alpha)$.\qed

\end{document}